\documentclass[11pt]{amsart}
\usepackage[margin=1.5in]{geometry}
\usepackage{cite}
\usepackage[utf8]{inputenc}
\usepackage[T1]{fontenc}
\usepackage[english]{babel}
\usepackage{amsmath,amsfonts,amsthm, mathrsfs}
\usepackage{hyperref}
\usepackage{blkarray}

\allowdisplaybreaks

\theoremstyle{plain}
\newtheorem{theorem}{Theorem}[section]

\newtheorem{lemma}[theorem]{Lemma}
\newtheorem{proposition}[theorem]{Proposition}
\newtheorem{conjecture}[theorem]{Conjecture}
 
\theoremstyle{definition}

\newtheorem{example}[theorem]{Example}
\newtheorem{remark}[theorem]{Remark}

\numberwithin{equation}{section}

\DeclareMathOperator{\Prob}{P}

\DeclareMathOperator{\supp}{supp}

\title{Markov chains on finite fields with deterministic jumps}
\author{Jimmy He}
\address{Department of Mathematics, MIT, Cambridge, MA  02139}
\email{jimmyhe@mit.edu}

\begin{document}
\maketitle
\begin{abstract}
    We study the Markov chain on $\mathbf{F}_p$ obtained by applying a function $f$ and adding $\pm\gamma$ with equal probability. When $f$ is a linear function, this is the well-studied Chung--Diaconis--Graham process. We consider two cases: when $f$ is the extension of a non-linear rational function which is bijective, and when $f(x)=x^2$. In the latter case, the stationary distribution is not uniform and we characterize it when $p=3\pmod{4}$. In both cases, we give an almost linear bound on the mixing time, showing that the deterministic function dramatically speeds up mixing. The proofs involve establishing bounds on exponential sums over the union of short intervals.
\end{abstract}

\section{Introduction}
Recent work of Chatterjee and Diaconis studied how Markov chains can be sped up by interspersing a deterministic function between steps of the Markov chain \cite{CD20}. For "most" bijective functions, the mixing time becomes logarithmic in the size of the state space. They asked for concrete examples where this speedup takes place. Currently, one of the only known examples is the Chung--Diaconis--Graham process, which is a Markov chain on the finite field $\mathbf{F}_p$ given by
\begin{equation}
\label{eq: cdg}
    X_{n+1}=aX_n+\varepsilon_{n+1},
\end{equation}
where the $\varepsilon_n$ are uniform on $\{-1,0,1\}$ and independent. While the walk $X_{n+1}=X_n+\varepsilon_{n+1}$ mixes in order $p^2$ steps, for certain $a\in\mathbf{F}_p$ and almost all $p$, the walk defined by \eqref{eq: cdg} mixes in order $\log(p)$ steps \cite{EV20}.

The purpose of this paper is to study the mixing times of non-linear analogues of \eqref{eq: cdg}. In particular, for bijections which are extensions of rational functions on $\mathbf{F}_p$, we are able to show that the mixing time is faster than $p^{1+\varepsilon}$ for any $\varepsilon>0$. This provides more examples of Markov chains which mix faster after adding deterministic jumps. We also consider the case when $f(x)=x^2$. We consider a slightly different version of the Markov chains considered in \cite{CD20}, and so we also explain how to apply our methods to the Markov chains appearing there.

\subsection{Main results}
Let us now define the class of functions $f$ that we will work with. For a prime $p$ and $d\in\mathbf{N}$, let $\mathcal{B}(p,d)$ denote the set of functions $f:\mathbf{F}_p\to\mathbf{F}_p$ which are bijections, and for which there exist coprime $P,Q\in\mathbf{F}_p[x]$ polynomials of degree at most $d$ such that $f=P/Q$ except at the zeroes of $Q$, and such that $P/Q$ is not a linear or constant function. 

\begin{remark}
Some examples of functions $f$ to keep in mind are permutation polynomials, which are polynomials $P\in\mathbf{F}_p[x]$ that are also bijections, and the function $f(x)=x^{-1}$ for $x\neq 0$ and $f(0)=0$. A specific family of permutation polynomials is given by $f(x)=x^3$ for primes $p>3$ with $(3,p-1)=1$.
\end{remark}

\begin{theorem}
\label{thm: bijective}
Let $p$ be a prime and $d\in\mathbf{N}$. Let $f\in\mathcal{B}(p,d)$, and let $\gamma\in\mathbf{F}_p$ be non-zero. Let $\varepsilon_i$ denote independent random variables uniform on $\pm\gamma$. Consider the lazy Markov chain defined by
\begin{equation*}
    X_{n+1}=\begin{cases}f(X_n)+\varepsilon_{n+1} &\text{ with probability }1/2,
    \\ X_n &\text{ with probability }1/2.
    \end{cases}
\end{equation*}
Then if $\pi$ denotes the uniform distribution on $\mathbf{F}_p$, for any $\varepsilon>0$, there is some constant $C=C(\varepsilon,d)$ such that
\begin{equation*}
    \sup_{X_0\in\mathbf{F}_p}\|\Prob(X_n\in\cdot)-\pi\|_{TV}\leq e^{-C\frac{n}{p^{1+\varepsilon}}}.
\end{equation*}
\end{theorem}

\begin{remark}
The \emph{mixing time} $t_{mix}(\varepsilon)$ of an ergodic Markov chain $P$ is the minimum $n$ such that
\begin{equation*}
    \sup_{X_0\in\mathbf{F}_p}\|\Prob(X_n\in\cdot)-\pi\|_{TV}\leq \varepsilon,
\end{equation*}
where $\pi$ is the stationary distribution. The total variation bound in Theorem \ref{thm: bijective} (and the corresponding bounds in Theorems \ref{thm: sq and add mixing time} and \ref{thm: bij other lazy}) imply that $t_{mix}(\varepsilon)$ is $O(p^{1+\delta})$ for any $\delta>0$, assuming that $f$ is the extension of rational functions of bounded degree.

We also remark that the $O(p^{1+\delta})$ could be improved to $O(p\log^Cp)$ for some large constant $C$ (probably $C=100$ would work). Since the bound we obtain is most likely far from optimal (see Remark \ref{rmk: lower bound}), we do not bother to explicitly state the stronger bound.
\end{remark}

\begin{remark}
	The assumption that the $\varepsilon_i$ are uniform on $\pm \gamma$ is easily relaxed. By considering the bijection $f+\gamma_0$, we can choose the $\varepsilon_i$ to be supported on any two point set. Since the bound from Cheeger's inequality can only get worse as edges are removed, this extends to any distribution $\mu$ for the $\varepsilon_i$ as long as $\mu$ is supported on at least two points, with the constant depending on $\min_{\gamma\in \supp \mu} \mu(\gamma)$.
	
	The laziness can also be removed at the cost of some assumptions on the distribution of the $\varepsilon_i$, see Theorem \ref{thm: bij other lazy}.
\end{remark}

\begin{remark}
This random walk is a slight modification of the walks considered in \cite{CD20}, which always applies the function $f$, but possibly adds $0$ instead of $\pm \gamma$. A small modification to the argument allows a similar result to be established in this setting as well, see Theorem \ref{thm: bij other lazy}. This gives a large class of examples where adding a deterministic bijection dramatically speeds up the mixing time. Finding such examples was a question raised in \cite{CD20}.
\end{remark}
\begin{remark}
\label{rmk: lower bound}
By an entropy argument, it's easy to see that at least order $\log(p)$ steps are necessary in order for the random walk to be close to uniform. This leaves a large gap between the upper and lower bounds and determining the correct order of the mixing time seems to be a challenging problem.
\end{remark}

The proof of Theorem \ref{thm: bijective} involves applying a version of Cheeger's inequality for directed graphs to reduce the problem to giving a lower bound for the number of solutions to $y=f(x)\pm \gamma$ with $x\in S$ and $y\in S^c$. This bound is obtained using some ideas from analytic number theory, namely exponential sum bounds, with a key ingredient being the Weil bound. While applications of the Weil bound have previously appeared in the study of expanders (see \cite{C89} for example), our approach does not require the eigenvalues to be exactly identified. This provides a more flexible approach, but gives a weaker bound.

The main technical tools are exponential sum bounds of the form
\begin{equation}
\label{eq: exp sum bd intro}
\sum_{j=1}^J\left|\sum_{x\in I_j} \exp(2\pi i kf(x)/p)\right|\leq Cp\log^2(L),
\end{equation}
where the $I_j$ are disjoint intervals of length $L$, and $C$ is some constant depending on $f$, along with bounds for linear exponential sums. These bounds appear to be new, although the proof of \eqref{eq: exp sum bd intro} closely follows ideas of Browning and Haynes \cite{BH13}, who proved the bound when $f(x)=1/x$.

Finally, we also give some results for the square-and-add Markov chain, which is the case when $f(x)=x^2$. This is obviously not a bijection, but when $p=3\pmod{4}$, the stationary distribution can be determined and a similar mixing time bound can be established.

\begin{theorem}
\label{thm: sq and add}
Let $p$ be a prime, $p=3\pmod{4}$, and let $\gamma\in\mathbf{F}_p$ be non-zero. Let $\varepsilon_i$ denote independent random variables uniform on $\pm\gamma$. Define a Markov chain on $\mathbf{F}_p$ by $X_{n+1}=X_n^2+\varepsilon_{n+1}$. Then the chain has a unique recurrent communicating class, and thus a unique stationary distribution. The stationary distribution is given by
\begin{equation*}
    \pi(\alpha)=\frac{|\{\beta\in\mathbf{F}_p\mid \beta^2\pm\gamma=\alpha\}|}{2p}.
\end{equation*}
Furthermore, the Markov chain is aperiodic if $\gamma=1$.
\end{theorem}

\begin{theorem}
\label{thm: sq and add mixing time}
Let $p$ be a prime, $p=3\pmod{4}$, and let $\gamma\in\mathbf{F}_p$ be non-zero. Let $\varepsilon_i$ denote independent random variables uniform on $\pm\gamma$. Consider the lazy Markov chain defined by
\begin{equation*}
    X_{n+1}=\begin{cases}X_n^2+\varepsilon_{n+1} &\text{ with probability }1/2,
    \\ X_n &\text{ with probability }1/2.
    \end{cases}
\end{equation*}
Then if $\pi$ denotes the stationary distribution on $\mathbf{F}_p$, for any $\varepsilon$, there is some constant $C=C(\varepsilon)$ such that
\begin{equation*}
    \sup_{X_0\in\mathbf{F}_p}\|\Prob(X_n\in\cdot)-\pi\|_{TV}\leq e^{-C\frac{n}{p^{1+\varepsilon}}}.
\end{equation*}
\end{theorem}

These results partially answer some questions raised in \cite{DHI20}. The case of $p=1\pmod{4}$ seems to be more mysterious and even the stationary distribution is not well-understood. Some heuristics and a conjecture are given in Section \ref{sec: p=1 mod 4}.

\subsection{Related work}
The Markov chains considered in this paper can be viewed as a non-linear analogue of the Chung--Diaconis--Graham process, defined by $X_{n+1}=aX_n+\varepsilon_{n+1}$ on $\mathbf{F}_p$ (or more generally on $\mathbf{Z}/(p)$ for composite $p$). This was first introduced in \cite{CDG87}, and studied in \cite{H06, H09, N11, H19, BV19}. Recent work of Eberhard and Varj\'u established cutoff for this Markov chain for many values of $a$ \cite{EV20}. Specifically, they show that when $a=2$, the Markov chain has cutoff at $c\log_2 p$ for an explicit $c\approx 1.01136$ for almost all $p$, and extend this to many other values of $a$. These works all rely heavily on the linear nature of the problem.

Chatterjee and Diaconis studied how deterministic jumps can increase the speed of convergence for Markov chains \cite{CD20}. They showed that for a state space of size $n$, applying a (fixed) random bijection gives a mixing time of order $\log(n)$ with high probability, but noted that specific examples were hard to find. The random walks treated in this paper provide some progress in this direction. This fits into a broader theme of finding ways to accelerate the convergence of Markov chains to their stationary distribution \cite{HSS20, BQZ20, ABLS07}.

The Markov chain when $f(x)=x^2$ studied in this paper can be viewed as a random version of a discrete dynamical system on $\mathbf{F}_p$. The dynamical system defined by iterating the map $x\mapsto x^2+c$ has been well-studied although many questions remain, see \cite{PMMY01, VS04, KLMMSS16}.

The exponential sum bounds are obtained using arguments that closely follows that of \cite{BH13}, who considered the case when $f(x)=1/x$. Their work is based on an argument Heath-Brown \cite{HB13}. These arguments work for any finite field, not just prime fields. It would be interesting to see if these ideas could be useful in studying the mixing times of random walks on $\mathbf{F}_q$, for $q$ a prime power.

\subsection{Outline}
Section \ref{sec: sq and add} establishes some basic properties of the square-and-add Markov chain, including the proof of Theorem \ref{thm: sq and add}. In Section \ref{sec: exp sum bounds}, we prove the exponential sum bounds needed, which are the main technical tool. Finally, in Section \ref{sec: mix bound}, we apply the bounds from Section \ref{sec: exp sum bounds} to prove Theorem \ref{thm: bijective}, and explain how to extend the arguments to establish Theorem \ref{thm: sq and add mixing time} and to study some of the random walks defined in \cite{CD20}.

\subsection{Notation}
Let $e_p(x)=\exp(2\pi i x/p)$. We will use $C$ to denote a constant that may change from line to line.

\section{The square-and-add Markov chain}
\label{sec: sq and add}
In this section, we establish some basic properties of the Markov chain when $f(x)=x^2$. The reason this case is more involved is that $f$ is not a bijection, and so even the stationary distribution is not obvious. While the case of $p=1\pmod{4}$ seems wild, surprisingly the case of $p=3\pmod{4}$ is tractable. 

\subsection{The case \texorpdfstring{$p=3\pmod 4$}{p=3 mod 4}}
The reason that the case of $p=3\pmod{4}$ is much easier is that in this case, exactly one of $\alpha$ and $-\alpha$ is a quadratic residue mod $p$. Using this fact, we are able to prove all essential properties about the Markov chain.

\begin{lemma}
\label{lem: stationary distr}
A stationary distribution for the Markov chain defined in Theorem \ref{thm: sq and add} is given by
\begin{equation*}
    \pi(\alpha)=\frac{|\{\beta\in\mathbf{F}_p\mid \beta^2\pm\gamma=\alpha\}|}{2p}.
\end{equation*}
\end{lemma}
\begin{proof}
The result can be shown by checking that
\begin{equation*}
    \sum_{\alpha^2\pm \gamma=\beta}\frac{1}{2}\pi(\alpha)=\pi(\beta)
\end{equation*}
for all $\beta\in\mathbf{F}_p$, which is equivalent to showing that
\begin{equation}
\label{eq: stationary equation}
    \sum_{\alpha^2\pm\gamma=\beta}|\{\rho\mid \rho^2\pm \gamma=\alpha\}|=2|\{\rho\mid\rho^2\pm \gamma=\beta\}|.
\end{equation}
Note that
\begin{equation*}
    |\{\rho\mid\rho^2\pm \gamma=\alpha\}|+|\{\rho\mid\rho^2\pm \gamma=-\alpha\}|=4
\end{equation*}
because exactly one of $\alpha+\gamma$ and $-\alpha-\gamma$ is a quadratic residue, and similarly for $\alpha-\gamma$ and $-\alpha+\gamma$ (if $\alpha=\pm\gamma$, the equation also holds since the solution to $\rho^2=0$ is counted twice). Since the sum in $\eqref{eq: stationary equation}$ is over pairs $\pm\alpha$, this implies \eqref{eq: stationary equation} holds for $\beta\neq \pm \gamma$. 

If $\beta=\pm\gamma$, the only difference is there is an extra term $|\{\rho\mid\rho^2=\pm \gamma\}|=2$ in the left hand side of \eqref{eq: stationary equation} corresponding to $\alpha=0$ which is not paired. But of course this corresponds to the solution to $\rho^2=0$ on the right hand side of \eqref{eq: stationary equation}.

\end{proof}

\begin{lemma}
\label{lem: ergodic}
The Markov chain defined in Theorem \ref{thm: sq and add} has a unique recurrent communicating class, and the restriction to this component is aperiodic if $\gamma=1$.
\end{lemma}
\begin{proof}
Any recurrent state must be of the form $\alpha^2\pm \gamma$, and so must be contained in $\supp\pi$. Conversely, since $\pi$ is a stationary distribution, every element of $\supp\pi$ is recurrent. It thus suffices to show that $\supp \pi$ contains a single communicating class. We do so as follows.

We claim for any $\alpha,\beta\in \supp\pi$, we can find a sequence of moves $\alpha\mapsto -\alpha$ and $\alpha\mapsto \alpha\pm 2\gamma$ taking us from $\alpha$ to $\beta$, while staying entirely inside $\supp\pi$. We'll then show that these moves stay within a single communicating class, from which we can conclude that $\supp\pi$ contains a single communicating class.

To see that we can go from $\alpha$ to $\beta$ using these moves while staying in $\supp\pi$, we argue as follows. At least one of $\alpha$ and $-\alpha$ lie in $\supp \pi$, since either $\alpha+\gamma$ or $-\alpha-\gamma$ is a quadratic residue as $p=3\pmod{4}$. Similarly, either $\alpha,\alpha+2\gamma\in\supp \pi$, or $-\alpha,-\alpha-2\gamma\in \supp \pi$, as one of $\alpha+\gamma$ and $-\alpha-\gamma$ will be a quadratic residue. A similar statement holds for $\alpha$ and $\alpha-2\gamma$. We thus partition $\mathbf{F}_p$ into sets $\{\pm \alpha\}$ and $\{0\}$, and note that each set contains an element of $\supp\pi$. Starting from $\alpha\in \supp\pi$, we can then move to an element in $\{\pm (\alpha+2\gamma)\}$ while staying inside $\supp\pi$, by either moving directly from $\alpha$ to $\alpha+2\gamma$ if it lies in $\supp\pi$, or moving from $\alpha$ to $-\alpha$, and then $-\alpha-2\gamma$ which are both guaranteed to lie in $\supp\pi$ if $\alpha+2\gamma$ does not. Similarly, we can move to an element in $\{\pm(\alpha-2\gamma)\}$. In this way, we can reach an element in any set of the form $\{\pm(\alpha+2k\gamma)\}$. But this exhausts $\mathbf{F}_p$, since $2\gamma$ is non-zero as $p\neq 2$, so we can reach either $\beta$ or $-\beta$. If we reach $-\beta$, then as $\beta$ by assumption also lies in $\supp\pi$, we can move from $-\beta$ to $\beta$.

Now we show the claims that if $\alpha,-\alpha\in \supp\pi$, then $\alpha$ and $-\alpha$ belong to the same communicating class, and that if $\alpha,\alpha+2\gamma\in \supp\pi$, then they belong to the same communicating class. 

Suppose that $\alpha$ and $-\alpha$ both lie in $\supp \pi$. Then as there is a path from $\alpha$ to $\alpha^2+\gamma$, and $\alpha\in \supp \pi$ so it must be recurrent, there must be a path from $\alpha^2+\gamma$ to $\alpha$. But then there is a path from $-\alpha$ to $\alpha^2+\gamma$, and then $\alpha$. Thus, $\alpha$ and $-\alpha$ must belong to the same communicating class.

Now suppose that $\alpha$ and $\alpha+2\gamma$ both lie in $\supp\pi$. If $\alpha+\gamma$ is a quadratic residue, then note that one of its square roots lies in $\supp\pi$ (it was already shown that one of $\alpha$ and $-\alpha$ lie in $\supp\pi$ for any $\alpha$). This square root must then be recurrent, and so from $\alpha$, we must be able to reach it, after which we can move to $\alpha+2\gamma$. If $\alpha+\gamma$ is not a quadratic residue, then $-\alpha-\gamma$ is a quadratice residue, and so both $-\alpha$ and $-\alpha-2\gamma$ lie in $\supp\pi$ and by the same argument lie in the same communicating class. We can then use the fact that $-\alpha$ and $\alpha$ lie in the same communicating class, as well as $-\alpha-2\gamma$ and $\alpha+2\gamma$, to conclude that $\alpha$ and $\alpha+2\gamma$ lie in the same communicating class.

To see that the random walk is aperiodic if $\gamma=1$, first note that there is a cycle of length $2$, namely $0\to 1\to 0$, and so it suffices to find an odd cycle. Now if a path from $\alpha^2+1$ to $\alpha$ were of even length, we would immediately be done because then after taking the step $\alpha\to \alpha^2+1$, we would obtain an odd cycle, so assume otherwise. That is, assume for all $\alpha$ that there is are only paths of odd length from $\alpha^2+1$ to $\alpha$. Then moving from $-\alpha$ to $\alpha^2+\gamma$, and then back to $\alpha$, is a path of even length. If $\alpha+1$ and is a quadratic residue, we obtain an odd length path going from $\alpha+2$ to $\sqrt{\alpha+1}$ and then moving to $\alpha$ gives an even length path from $\alpha+2$ to $\alpha$. If $-\alpha-1$ were a quadratic residue, we would obtain an even length path from $-\alpha-2$ to $-\alpha$. We can similarly assume paths from $\alpha^2-1$ to $\alpha$ are odd length, and obtain even length paths from either $\alpha-2$ to $\alpha$ or $-\alpha+2$ to $-\alpha$. But then using these we can construct an even length path from $\alpha^2+1$ back to $\alpha$, since we proved above that we can move from any element in $\supp\pi$ to any other using these moves. Thus, there must be an odd cycle.
\end{proof}

Theorem \ref{thm: sq and add} follows immediately from Lemmas \ref{lem: stationary distr} and \ref{lem: ergodic}.

\begin{example}
Consider the case $p=11$ and $\gamma=1$. Then the stationary distribution is given by
\begin{equation*}
    \pi=\left(\frac{2}{22},\frac{1}{22},\frac{4}{22},\frac{2}{22},\frac{4}{22},\frac{2}{22},\frac{2}{22},\frac{0}{22},\frac{2}{22},\frac{0}{22},\frac{3}{22}\right),
\end{equation*}
and it can be checked that the numerators are indeed given by counting solutions to $\beta^2\pm 1=\alpha$.
\end{example}

\subsection{The case \texorpdfstring{$p=1\pmod 4$}{p=1 mod 4}}
\label{sec: p=1 mod 4}
The case when $p=1\pmod 4$ is more wild, but there is some evidence to suggest that much of the behavior is similar to that of a random directed graph with certain degree constraints.

The following heuristic was suggested by Alex Cowan (private communication). For convenience, we work with the graph whose edges are given by $(\alpha, \beta)$ for $\beta=(\alpha\pm 1)^2$. Let $\pi'$ denote the stationary distribution of this walk. It's clear that the stationary distribution of the original walk can be recovered from the new one by $\pi(x)=\pi'(x-1)/2+\pi'(x+1)/2$.

The idea is that the degree distribution of the graph is determined, with about half the vertices having indegree $0$ and half having indegree $4$ (we just ignore $0$, $1$, and $-1$). Now the key observation is that the vertices with indegree $0$ come in pairs, as if $\alpha$ is not a quadratic residue, then neither is $-\alpha$. This means after removing all vertices with indegree $0$, if the edges were random subject to this constraint then we have $1/4$ of the vertices have indegree $0$, $1/2$ have indegree $2$ and $1/4$ have indegree $4$. If we assume that the resulting graph is random, then by applying work of Cooper and Frieze \cite{CF04}, we can obtain a precise conjecture for the size of the support.

\begin{conjecture}
Let $\varepsilon_n$ be uniformly distributed on $\pm 1$ and independent. Consider the Markov chain defined by $X_{n+1}=X_n^2+\varepsilon_{n+1}$ on $\mathbf{F}_p$ with $p=1\pmod{4}$. Then conjecturally it has a unique stationary distribution $\pi$, and
\begin{equation*}
    \frac{|\supp \pi|}{p}\to 1-\frac{1}{4}(1+\alpha)^2,
\end{equation*}
where $\alpha\approx 0.2956$ is the smallest positive root of $x^4+2x^2-4x+1$.
\end{conjecture}

This conjecture agrees with computer computations done by Steve Butler, which suggests that asymptotically the support of $\pi$ contains $58\%$ of the elements in $\mathbf{F}_p$.

\section{Counting solutions with exponential sum bounds}
\label{sec: exp sum bounds}
To obtain lower bounds on the Cheeger constant, after a reduction we will need to show that $f(x)\pm \gamma=y$ has many solutions for $x\in S$ and $y\in S'$, where $S$ and $S'$ are a disjoint union of arithmetic progressions. This will be shown following the argument of Browning--Haynes \cite{BH13}, who actually proved the needed result in the case of $f(x)=x^{-1}$ (the sum was restricted to $\mathbf{F}_p^\times$ but this is unimportant). We give the modifications needed to generalize to the class of functions $\mathcal{B}(p,d)$ as well as the function $f(x)=x^2$. The key is the Weil bound for exponential sums. We use the following convenient formulation given in \cite[Lemma 3]{GS10}.

\begin{lemma}
\label{lem: weil bound}
Let $P,Q\in\mathbf{F}_p[x]$ be coprime polynomials such that $P(x)/Q(x)$ is not a constant function on $\mathbf{F}_p$. Then
\begin{equation*}
    \left|\sum _{\substack{x\in\mathbf{F}_p\\Q(x)\neq 0}}e_p(P(x)/Q(x))\right|\leq Cp^{1/2},
\end{equation*}
where the constant $C$ depends only on the degree of $P$ and $Q$.
\end{lemma}

We now give bounds for averaged exponential sums over intervals. A separate elementary argument is given when $f(x)=x^2$.
\begin{lemma}
\label{lem: complete estimate rational}
Let $f\in\mathcal{B}(p,d)$. Then for any interval $I$, and any $k\in\mathbf{F}_p^\times$,
\begin{equation*}
    \sum_{n=1}^p\left|\sum_{x\in I}e_p(kf(x+n)) \right|^2\leq C p|I|,
\end{equation*}
where the constant $C$ depends only on $d$.
\end{lemma}
\begin{proof}
We have
\begin{equation*}
\begin{split}
    &\sum_{n=1}^p\left|\sum_{x\in I}e_p(kf(x+n)) \right|^2
    \\=&\sum_{x,y\in I}\sum_{n=1}^p e_p(k(f(x+n)-f(y+n)))
    \\=&\sum_{x,y\in I}\sum_{\alpha=1}^p\sum_{a,b=1}^p p^{-1} e_p(k(a-b))e_p(\alpha(f^{-1}(a)-f^{-1}(b)-x+y))
    \\=&p^{-1}\sum _{\alpha=1}^p\left|\sum_{a=1}^p e_p(\alpha a+k f(a))\right|^2\left|\sum_{x\in I}e_p(\alpha x)\right|^2,
\end{split}
\end{equation*}
where the second equality follows from the fact that the sum over $\alpha$ enforces the condition $f^{-1}(a)-f^{-1}(b)=x-y$ (taking $n=f^{-1}(a)-x=f^{-1}(b)-y$, this is equivalent to $a=f(x+n)$ and $b=(y+n)$) and the third equality uses that $f$ is bijective. Now $\alpha a+kf(a)$ is not constant as a function of $a$, because $f(a)$ is not linear. Then
\begin{equation*}
    \left|\sum_{a=1}^p e_p(\alpha a+k f(a))\right|\leq Cp^{1/2}
\end{equation*}
for a constant $C$ depending only on $d$, by Lemma \ref{lem: weil bound}. Since there are a bounded number of poles for $f$, the additional terms introduced by extending to $\mathbf{F}_p$ are negligible up to changing $C$. 

Note that for $\alpha\in [-p/2,p/2]$
\begin{equation*}
    \
    \left|\sum_{x\in I}e_p(\alpha x)\right|\leq \min(|I|,p/2|\alpha|).
\end{equation*}
Thus,
\begin{equation*}
\begin{split}
    p^{-1}\sum _{\alpha=1}^p\left|\sum_{a=1}^p e_p(\alpha a+k f(a))\right|^2\left|\sum_{x\in I}e_p(\alpha x)\right|^2&\leq p^{-1}C\sum_{|\alpha|\leq p/|I|} p|I|^2+p^{-1}C\sum_{|\alpha|>p/|I|}\frac{p^3}{\alpha^2}
    \\&\leq Cp|I|,
\end{split}
\end{equation*}
using the bound
\begin{equation*}
    \sum_{\alpha=n}^\infty \frac{1}{\alpha^2}\leq \frac{1}{n-1}.
\end{equation*}
\end{proof}

\begin{lemma}
\label{lem: complete estimate square}
For all $k\in\mathbf{F}_p^\times$, we have
\begin{equation*}
    \sum_{n=1}^p\left|\sum_{x\in I}e_p(k(x+n)^2) \right|^2=p|I|.
\end{equation*}
\end{lemma}
\begin{proof}
Expanding the square, we have
\begin{equation*}
\begin{split}
    \sum_{n=1}^p\left|\sum_{x\in I}e_p(k(x+n)^2) \right|^2&=\sum_{x,y\in I}\sum_{n=1}^p e_p(2k(x-y)n+k(x^2-y^2))
    \\&=p|I|.
\end{split}
\end{equation*}
\end{proof}

These averaged estimates feed into an argument originally due to Heath-Brown \cite{HB13}, to give the following bound. We include the proof for completeness.

\begin{proposition}
\label{prop: complete to incomplete}
Let $I_j$ be a collection of $J$ disjoint arithmetic progressions of length $L$, with common difference $\delta\in\mathbf{F}_p^\times$. Suppose that $f\in \mathcal{B}(p,d)$ or $f(x)=x^2$. Then for all $k\in\mathbf{F}_p^\times$,
\begin{equation*}
    \sum_{j=1}^J\left|\sum_{x\in I_j} e_p(kf(x))\right|^2\leq Cp\log^2(L+1),
\end{equation*}
where the constant depends only on $d$.
\end{proposition}

\begin{proof}
By replacing $f(x)$ with $f(\delta x)$ (or in the case of $f(x)=x^2$, absorbing $\delta^2$ into $k$), we may work with intervals $I_j$ instead of arithmetic progressions. The proof then proceeds exactly as in the proof of Theorem 2 in \cite{BH13} (this argument is originally due to Heath-Brown \cite{HB13}).

Let $S(n,h)=\sum_{x=n+1}^{n+h}e_p(kf(x))$. Let $x_j$ denote the first element in the interval $I_j$. Then the sum we are interested in is simply
\begin{equation*}
    \sum_{j=1}^J |S(x_j,L)|^2.
\end{equation*}
Now for any $1\leq l\leq L$ and $x_j-L<n\leq x_j$, we have
\begin{equation*}
    |S(x_j,l)|=|S(n, x_j-n+l)-S(n, x_j-n)|\leq 2\max _{l\leq 2L}|S(n,l)|,
\end{equation*}
and so
\begin{equation*}
    |S(x_j, L)|\leq \frac{2}{L}\sum_{x_j-L<n\leq x_j}\max_{l\leq 2L}|S(n,l)|.
\end{equation*}
Then by Cauchy's inequality,
\begin{equation*}
    |S(x_j, L)|^2\leq \frac{4}{L}\sum_{x_j-L<n\leq x_j}\max_{l\leq 2L}|S(n,l)|^2,
\end{equation*}
and summing over $j$ gives
\begin{equation*}
    \sum_{j=1}^J |S(x_j,L)|^2\leq \frac{4}{L}\sum_{n\in\mathbf{F}_p}\max_{l\leq 2L}|S(n,l)|^2,
\end{equation*}
where we use the fact that the intervals are disjoint. Now pick $t$ an integer so that $2L\leq 2^t\leq 4L$, and for each $n\in\mathbf{F}_p$, pick $l^*$ depending on $n$ so that $|S(n,l^*)|=\max_{l\leq 2L}|S(n,l)|$. Then writing the binary expansion $l^*=\sum _{s\in D}2^{t-s}$, we have
\begin{equation*}
    S(n,l^*)=\sum _{s\in D}S(n+v_{n,s}2^{t-s}, 2^{t-s}),
\end{equation*}
where
\begin{equation*}
    v_{n,s}=\sum_{u\in D, u<s}2^{s-u}<2^s.
\end{equation*}
Then we have
\begin{equation*}
\begin{split}
    |S(n,l^*)|^2&\leq |D|\sum_{s\in D}|S(n+v_{n,s}2^{t-s},2^{t-s})|^2
    \\&\leq (t+1)\sum_{0\leq  s\leq t}\sum_{0\leq v<2^s}|S(n+v2^{t-s},2^{t-s})|^2.
\end{split}
\end{equation*}
Finally, we can sum over $n$ and use Lemma \ref{lem: complete estimate rational} or Lemma \ref{lem: complete estimate square} to obtain
\begin{equation*}
\begin{split}
    \frac{4}{L}\sum_{n\in\mathbf{F}_p}\max_{l\leq 2L}|S(n,l)|^2&\leq \frac{Cp(t+1)}{L}\sum_{0\leq  s\leq t}\sum_{0\leq v<2^s}2^{t-s}
    \\&\leq \frac{Cp(t+1)^22^t}{L}
    \\&\leq Cp\log^2(L+1).
\end{split}
\end{equation*}

\end{proof}

We now prove a bound for linear exponential sums over a union of intervals.

\begin{proposition}
\label{prop: lin exp bound}
Let $I_j$ be a collection of $J$ disjoint arithmetic progressions of length $L$, with common difference $\delta\in\mathbf{F}_p^\times$. Then
\begin{equation*}
    \sum_{k=1}^{p-1}\left|\sum_{j=1}^J \sum_{x\in I_j}e_p(kx)\right|\leq CJ^{1/2}p\log^{3/2}(p),
\end{equation*}
for some constant $C>0$.
\end{proposition}
\begin{proof}
First, note that the sum is the same if we replace the arithmetic progressions $I_j$ by intervals $I_j/\delta$, so we may assume that $\delta=1$ and the $I_j$ are all intervals. Let $x_j$ denote the first element in each interval $I_j$, so $I_j=I_0+x_j$ for $I_0=\{1,2,\dotsc, L\}$. By relabeling the intervals if necessary, we will assume that the $x_j$ are ordered. Since the intervals are disjoint, if $j\neq j'$ then $|x_j-x_{j'}|\geq L$. 

The idea is to now break the sum over $k$ up into intervals of length $p/L$, and take advantage of the different scales at which the oscillations for the sums over $j$ and $x$ occur. For $1\leq l\leq (L-1)/2$ let $I_l'$ denote the interval $[lp/L,(l+1)p/L]$ and let $I_{-l}'=-I_l'$. Then
\begin{equation}
\label{eq: break up into intervals}
\begin{split}
    &\sum_{k=1}^{p-1}\left|\sum_{j=1}^J \sum_{x\in I_j}e_p(kx)\right|
    \\\leq &L\sum_{k=-p/L}^{p/L} \left|\sum_{j=1}^J e_p(kx_j)\right|+ \sum_{\substack{l=-(L-1)/2\\{l\neq 0}}}^{(L-1)/2}\frac{L}{|l|}\sum_{k\in I_l'} \left|\sum_{j=1}^J e_p(kx_j)\right|,
\end{split}
\end{equation}
where we use $\left|\sum_{x\in I_0}e_p(kx)\right|\leq \min(L,p/2|k|)$. Now note that for any interval $I'$ of length $p/L$,
\begin{equation}
\label{eq: CS short interval}
    \sum_{k\in I'}\left|\sum_{j=1}^J e_p(kx_j)\right|\leq (p/L)^{1/2}\left(\sum_{k\in I'}\left|\sum_{j=1}^J e_p(kx_j)\right|^2\right)^{1/2}.
\end{equation}
Now
\begin{equation*}
\begin{split}
    \sum_{k\in I'}\left|\sum_{j=1}^J e_p(kx_j)\right|^2&=\sum_{j,j'=1}^J\sum_{k\in I'}e_p(k(x_j-x_{j'}))
    \\&\leq \sum_{j,j'=1}^J\min(p/L,p/|x_j-x_{j'}|)
\end{split}
\end{equation*}
where $x_j-x_{j'}$ is the representative mod $p$ between $-p/2$ and $p/2$. But now by the spacing condition and the ordering of the $x_j$, we have $|x_j-x_{j'}|\geq |j-j'|L$ (where again $j-j'$ is the representative between $-J/2$ and $J/2$ mod $J$). Thus,
\begin{equation*}
    \sum_{j,j'\in J}\min(p/L,p/|x_j-x_{j'}|)\leq \frac{Jp}{L} +2\sum_{m=1}^{J/2}\sum_{|j-j'|=m}\frac{p}{mL}\leq 3\frac{Jp}{L}\log(p),
\end{equation*}
and combined with \eqref{eq: CS short interval} this gives
\begin{equation*}
    \sum_{k\in I'}\left|\sum_{j=1}^J e_p(kx_j)\right|\leq 3\frac{pJ^{1/2}}{L}\log^{1/2}(p).
\end{equation*}
Finally, this together with \eqref{eq: break up into intervals} gives the desired inequality.
\end{proof}

Using the exponential bounds established in this section, we can prove the following estimate for the number of solutions to $f(x)=y$ with $x\in S$ and $y\in S'$ when $S$ and $S'$ are the union of arithmetic progressions.

\begin{proposition}
\label{prop: main sol estimate}
Let $S\subseteq \mathbf{F}_p$ be a disjoint union of $J$ arithmetic progressions $I_j$ of length $L$ and common difference $\delta\in\mathbf{F}_p^\times$, and let $S'$ be a disjoint union of $J$ arithmetic progressions of length $L'$ and common difference $\delta$. Suppose that $f\in \mathcal{B}(p,d)$ or $f(x)=x^2$, and suppose that $JLL'\geq p^{3/2+\varepsilon}$ for some $\varepsilon>0$. Then for large enough $p$,
\begin{equation*}
    |\{x\in S|f(x)\in S'\}|\geq c\frac{|S||S'|}{p},
\end{equation*}
where $c>0$ is some constant depending only on $\varepsilon$ and $d$.
\end{proposition}
\begin{proof}
We have
\begin{equation*}
\begin{split}
    |\{x\in S|f(x)\in S'\}|&=\frac{1}{p}\sum_{k=1}^p \sum _{x\in S}\sum_{y\in S'} e_p(k(f(x)-y))
    \\&=\frac{|S||S'|}{p}+R,
\end{split}
\end{equation*}
where
\begin{equation*}
    R=\frac{1}{p}\sum_{k=1}^{p-1}\sum _{x\in S}\sum_{y\in S'} e_p(k(f(x)-y)).
\end{equation*}
By the triangle inequality,
\begin{equation*}
    |R|\leq \frac{1}{p}\sum_{k=1}^{p-1}\left|\sum_{x\in S}e_p(kf(x))\right|\left|\sum_{y\in S'}e_p(-ky)\right|.
\end{equation*}
Now by Cauchy--Schwarz,
\begin{equation*}
    \left|\sum_{x\in S}e_p(kf(x))\right|\leq \sum_{j=1}^J\left|\sum_{x\in I_j}e_p(kf(x))\right|\leq J^{1/2}\left(\sum_{j=1}^J\left|\sum_{x\in I_j}e_p(kf(x))\right|^2\right)^{1/2}.
\end{equation*}
Since $k\neq p$, Proposition \ref{prop: complete to incomplete} implies,
\begin{equation*}
    \left|\sum_{x\in S}e_p(kf(x))\right|\leq J^{1/2}Cp^{1/2}\log(L).
\end{equation*}
By Proposition \ref{prop: lin exp bound},
\begin{equation*}
    \sum_{k=1}^{p-1}\left|\sum_{y\in S'}e_p(-ky)\right|\leq CJ^{1/2}p\log^{3/2}(p).
\end{equation*}
Thus,
\begin{equation*}
    |R|\leq CJp^{1/2}\log^{5/2}(p).
\end{equation*}
Now if $JLL'\geq p^{3/2+\varepsilon}$, then $|S||S'|/p\geq Jp^{1/2+\varepsilon}$ which dominates $R$. Finally, as $C$ depends only on $d$, we can choose $c>0$, depending only on the degree $d$ and $\varepsilon$, so that for large enough $p$, we have $|\{x\in S|f(x)\in S'\}|\geq c|S||S'|/p$.

\end{proof}

\section{Mixing time bounds}
\label{sec: mix bound}
The mixing time bounds for the cases when $f\in\mathcal{B}(p,d)$ and when $f(x)=x^2$ are both established using the same overall argument. This involves using Cheeger's inequality for directed graphs to reduce to a problem of counting edges between two sets, and then utilizing the results in Section \ref{sec: exp sum bounds} to show that there are many edges. The argument for the case when $f(x)=x^2$ is complicated by the fact that the stationary distribution does not have full support, and we explain how to work around this.

Let $P$ be an ergodic Markov chain on $X$ with stationary distribution $\pi$. Define the \emph{Cheeger constant} $h(P)$ by
\begin{equation*}
    h(P)=\min_{S\subseteq X}\frac{\sum_{x\in S, y\in S^c}\pi(x)P(x,y)}{\min(\pi(S),\pi(S^c))},
\end{equation*}
where $S$ ranges over all non-trivial subsets of $X$. Note that if $P$ is simple random walk on a $k$-regular (i.e. both in and out-degrees are $k$) directed graph $G$, then this definition reduces to
\begin{equation}
\label{eq: cheeger regular}
    h(P)=\min _{S\subseteq X}\frac{e(S,S^c)}{k\min(|S|,|S^c|)},
\end{equation}
where $e(S,S^c)$ is the number of edges going from $S$ to $S^c$ in $G$. When $f$ is a bijection, our random walk falls into this case. Note also that because $\pi$ is the stationary distribution,
\begin{equation*}
    \sum_{x\in S, y\in S^c}\pi(x)P(x,y)=\sum_{y\in S, x\in S^c}\pi(x)P(x,y),
\end{equation*}
and in particular for a random walk on a regular directed graph, $e(S,S^c)=e(S^c,S)$. Thus, we may consider edges in both directions up to a factor of $2$.

The key tool for the bounds on the mixing times is Cheeger's inequality for non-reversible Markov chains. The following theorem follows from some standard facts about the spectral theory of Markov chains. See for example, page 23 of \cite{MT06}, noting that $(I+P)/2$ is lazy, and $h((I+P)2)=h(P)/2$ (see also \cite{C05} for the special case of random walk on a directed graph).

\begin{theorem}
Let $P$ be an ergodic Markov chain on $X$ with stationary distribution $\pi$. Consider the lazy chain $X_n$ with transition matrix $(I+P)/2$, and starting from some deterministic $X_0$. Then if 
\begin{equation*}
    n\geq 4h(P)^{-2}(\max_{x\in X} \log(\pi(x)^{-1})+2c)
\end{equation*}
for some $c>0$, we have
\begin{equation*}
    \sup_{X_0\in X}\|\Prob(X_n\in\cdot)-\pi\|_{TV}\leq e^{-c}.
\end{equation*}
\end{theorem}

\subsection{Proof of Theorem \ref{thm: bijective}}
\label{sec: pf of bij}
We show that in the supremum over $S$, the set $S$ can be assumed to have some structure, at the cost of a constant. Since we are only interested in the order of the mixing time, the loss of a constant is okay.

First, we define a decomposition of any set $S\subseteq \mathbf{F}_p$ into certain types of arithmetic progressions. We will call an arithmetic progression with difference $d$ a \emph{$d$-AP}. A \emph{$2\gamma$-AP decomposition} of $S$ is a decomposition $S=\bigsqcup I_k$ where the $I_k$ are $2\gamma$-APs, the number of $I_k$'s is minimal. Such a decomposition always exists and is unique.

It is easy to see that if $S$ has a $2\gamma$-AP decomposition into $J$ $2\gamma$-APs, then the same is true of $S^c$.

We now show that when computing $\min _{S}\frac{e(S,S^c)}{\min(|S|,|S^c|)}$, we may assume that there are at most $|S|p^{-1/2-\varepsilon/2}$ $2\gamma$-APs in the $2\gamma$-AP decomposition of $|S|$. Thus, let $\mathcal{P}$ denote the set of $S$ whose $2\gamma$-AP decomposition contains at most $|S|p^{-1/2-\varepsilon/2}$ $2\gamma$-APs.
\begin{lemma}
\label{lem: bij red to int}
We have for large enough $p$,
\begin{equation*}
    \min _{S}\frac{e(S,S^c)}{\min(|S|,|S^c|)}\geq \min\left(p^{-1/2-\varepsilon/2},\min _{S\in \mathcal{P}}\frac{e(S,S^c)}{\min(|S|,|S^c|)}\right).
\end{equation*}
\end{lemma}
\begin{proof}
Assume that $|S|\leq |S^c|$. Write $S=\bigsqcup I_k$ the $2\gamma$-AP decomposition. Note that each right endpoint of a 2$\gamma$-AP contributes an edge from either $S$ to $S^c$ or from $S^c$ to $S$, because if $x$ is a right endpoint, then $x+2\gamma\in S^c$ and $x$ and $x+2\gamma$ are both connected to $f^{-1}(x+\gamma)$. If there were more than $|S|p^{-1/2-\varepsilon/2}$ many $2\gamma$-APs, then this would imply
\begin{equation*}
    \frac{e(S,S^c)}{\min(|S|,|S^c|)}\geq p^{-1/2-\varepsilon/2}.
\end{equation*}
\end{proof}

The $2\gamma$-AP decomposition of sets in $\mathcal{P}$ have arithmetic progressions of different lengths. The following lemma lets us reduce to the case when all arithmetic progressions have the same length.

\begin{lemma}
\label{lem: subint same length}
Let $S=\bigsqcup_{j=1}^J I_j\subseteq \mathbf{F}_p$ be a disjoint union of arithmetic progressions with common difference $\delta\in\mathbf{F}_p^\times$. Let $L=|S|/J$ denote the average length of the $I_j$. Then $S$ contains $J$ disjoint arithmetic progressions $I_j'$ with common difference $\delta$ and length $\lfloor L/4\rfloor$.
\end{lemma}
\begin{proof}
First, note that it suffices to prove the result for intervals. We have
\begin{equation*}
    \sum_{|I_j|\geq L/2}|I_j|\geq |S|/2,
\end{equation*}
and by splitting the intervals $I_j$ with $|I_j|\geq L/2$ up into intervals of length exactly $\lfloor L/4\rfloor$, we throw away at most $JL/4$ points. Thus, we obtain disjoint intervals of length $\lfloor L/4\rfloor$, whose total length is at least $|S|/4$, and so there must be at least $J$ such intervals.
\end{proof}

We are now in a position to prove Theorem \ref{thm: bijective} using the bounds given by Proposition \ref{prop: main sol estimate}.
\begin{proof}[Proof of Theorem \ref{thm: bijective}]
By Cheeger's inequality, it suffices to show that the Cheeger constant is bounded from below by $p^{-1/2-\varepsilon/2}$. By Lemma \ref{lem: bij red to int}, it suffices to give a lower bound for 
\begin{equation*}
    \min _{S\in \mathcal{P}}\frac{e(S,S^c)}{\min(|S|,|S^c|)}
\end{equation*}
of order $p^{-1/2-\varepsilon/2}$ (in fact when $S\in\mathcal{P}$ we'll give a constant order lower bound). Thus, take $S\in\mathcal{P}$ and assume that $|S|\leq |S^c|$. 

Using Lemma \ref{lem: subint same length}, we can find subsets $S_1\subseteq S$ and $S_2\subseteq S^c$ with $|S_1|\geq |S|/4$ and $|S_2|\geq |S^c|/4$ which are a disjoint union of at most $J\leq |S|p^{-1/2-\varepsilon/2}$ $2\gamma$-APs of fixed lengths $L_1$ and $L_2$.

Now $J\leq |S|p^{-1/2-\varepsilon/2}$ and $JL_1=|S_1|\geq |S|/4$, and so $L_1\geq p^{1/2+\varepsilon/2}/4$. Similarly, $JL_2\geq |S^c|/4\geq p/8$. Then $JL_1L_2\geq p^{3/2+\varepsilon/2}/32$, and so Proposition \ref{prop: main sol estimate} shows that
\begin{equation*}
    e(S,S^c)\geq e(S_1,S_2)\geq c\frac{|S||S^c|}{p}
\end{equation*}
for some positive constant $c$ depending only on $d$ and $\varepsilon$. This implies that
\begin{equation*}
    \frac{e(S,S^c)}{\min(|S|,|S^c|)}\geq c\frac{|S^c|}{p}\geq \frac{c}{2}.
\end{equation*}
Since the estimate is uniform over $S\in\mathcal{P}$, this completes the proof.
\end{proof}

\subsection{Proof of Theorem \ref{thm: sq and add mixing time}}
We now sketch the adjustments that have to be made to handle the case of $f(x)=x^2$. The main difficulty lies in the fact that the support of the stationary distribution $\supp\pi_p$ is not all of $\mathbf{F}_p$. To adapt the argument from the bijective case, we take advantage of the fact that when $p=3\pmod{4}$, at least one of $\alpha$ and $-\alpha$ lies in $\supp\pi$, and that $\pi$ is close to uniform on its support.

Consider now the random walk restricted to $\supp\pi$. Since the walk enters this set after a single step, it suffices to bound the mixing time for the restricted walk. Now although the walk is not a simple random walk on a regular graph, the transition probabilities are of constant order in $p$ and by Theorem \ref{thm: sq and add} the stationary distribution is essentially constant up to a factor of $4$. Thus, \eqref{eq: cheeger regular} holds up to some constant factor, and so we provide a lower bound for $e(S,S^c)/\min(|S|,|S^c|)$.

We now adapt the $2\gamma$-AP decomposition to this setting. A set $S\subseteq \supp\pi$ is \emph{symmetric} if it is of the form $\widetilde{S}\cap \supp\pi$ for a set $\widetilde{S}\subseteq \mathbf{F}_p$ with $\widetilde{S}=-\widetilde{S}$. A \emph{symmetric $2\gamma$-AP} is a set of the form $(J\cup -J)\cap \supp \pi$, where $J$ is a $2\gamma$-AP contained in $\{0,1,\dotsc, (p-1)/2\}$. A \emph{$2\gamma$-AP decomposition} of a symmetric set $S\subseteq \supp\pi$ is a decomposition $S=\bigsqcup I_k$ where the $I_k$ are of the form $I_k=\supp\pi\cap (J_k\cup -J_k)$, and the $J_k$ are $2\gamma$-APs contained in $\{0,1,\dotsc, (p-1)/2\}$, such that the number of $I_k$'s is minimal. It is clear that this exists and is unique, and moreover the $2\gamma$-APs $J_k$ are uniquely determined as well.

The idea is to use this as a replacement for the $2\gamma$-AP decomposition used in the proof for $f\in\mathcal{B}(p,d)$.

Let $\mathcal{P}$ denote the set of subsets $S\subseteq \supp\pi$ which are symmetric, and whose $2\gamma$-AP decomposition contains at most $J\leq |S|p^{-1/2-\varepsilon/2}$ $2\gamma$-APs.

\begin{lemma}
We have for large enough $p$,
\begin{equation*}
    \min _{S}\frac{e(S,S^c)}{\min(|S|,|S^c|)}\geq \max\left(p^{-1/2-\varepsilon/2},\frac{1}{16}\min _{S\in \mathcal{P}}\frac{e(S,S^c)}{\min(|S|,|S^c|)}\right).
\end{equation*}
\end{lemma}
\begin{proof}
Assume that $|S|\leq |S^c|$, which can be done since $e(S,S^c)$ and $e(S^c,S)$ differ by at most a multiplicative constant, since $\pi$ is essentially constant on $\supp\pi$ up to a factor of $4$. Note that if $\alpha\in S$ and $-\alpha\in S^c$, then as both are connected to $\alpha^2\pm \gamma$, this contributes $2$ edges from $S$ to $S^c$. Thus, if there are more than $|S|^{1/2-\varepsilon/2}$ such pairs, we have
\begin{equation*}
    \frac{e(S,S^c)}{\min(|S|,|S^c|)}\geq |S|^{-1/2-\varepsilon/2}\geq p^{-1/2-\varepsilon/2}.
\end{equation*}

Otherwise, replace $S$ with the set $S'=S\cup \{\alpha\in S^c|-\alpha\in S\}$. Let $n=|S'|-|S|$, and note that $n\leq |S|^{1/2-\varepsilon/2}$ and $e(S,S^c)\geq 2n$. For large enough $p$, we have
\begin{equation}
\label{eq: denominator ineq}
    \min(|S|,|S^c|)=|S|\leq 2\min(|S'|,|(S')^c|),
\end{equation}
since either $\min(|S'|,|(S')^c|)=|S'|$ for which the statement obviously holds, or $|S|\geq |\supp\pi|/2-n$, in which case we have
\begin{equation*}
	|(S')^c|=|S^c|-n\geq |S|-|S|^{1/2-\varepsilon/2},
\end{equation*}
and for large enough $p$ this also implies \eqref{eq: denominator ineq}.

The numerator increases by at most $6n$, since each vertex has degree at most $6$, and so for large enough $p$
\begin{equation*}
    e(S',(S')^c)\leq e(S,S^c)+6n\leq 8e(S,S^c).
\end{equation*}
This implies that the minimum can be restricted to symmetric $S$, at the cost of the $1/16$ factor.

As in the proof of Lemma \ref{lem: bij red to int}, each right endpoint of a $2\gamma$-AP appearing in the decomposition of $S$ contributes an edge from $S$ to $S^c$. This is because if $x$ is the right endpoint of a $2\gamma$-AP $J$ appearing in the decomposition, one of $x+\gamma$ and $-x-\gamma$ is a quadratic residue, and thus either $x\in S$ and $x+2\gamma\in S^c$ contributes an edge, or $-x\in S$ and $-x-2\gamma\in S^c$ contributes an edge.
\end{proof}

With this replacement for Lemma \ref{lem: bij red to int}, we can now prove Theorem \ref{thm: sq and add mixing time}.

\begin{proof}[Proof of Theorem \ref{thm: sq and add mixing time}]
As in the proof of Theorem \ref{thm: bijective}, it suffices to give a lower bound for $e(S,S^c)/|S|$ when $S\in\mathcal{P}$ and $|S|\leq |S^c|$. A small extension of Lemma \ref{lem: subint same length} then reduces to sets $S_1$ and $S_2$ both of comparable size to $S$ and $S^c$ respectively, with $S_1$ and $S_2$ being the disjoint union of symmetric $2\gamma$-APs of lengths $L_1$ and $L_2$ respectively.

Note that if $J$ is a $2\gamma$-AP appearing in the decomposition of $S_1$, then if $x\in J$, at least one of $x$ and $-x$ belong to $S_1$. If there is an edge from $x\in J$ to $S_2$, there is also an edge from $-x$ to $S_2$, so if $S_1=\bigsqcup (J_j\cup -J_j)\cap \supp \pi$, then $e(S_1,S_2)\geq e(\bigsqcup J_j,S_2)$. Further, if $S_2=\bigsqcup (K_j\cup -K_j)\cap \supp\pi$, then any edge landing in $K_j$ must also land in $\supp \pi$, because $\supp\pi$ contains all elements of the form $\alpha^2\pm \gamma$, and so $e(\bigsqcup J_j, S_2)\geq e(\bigsqcup J_j,\bigsqcup K_j)$.

Since the sets have been written as a disjoint union of $2\gamma$-APs, we can now apply Proposition \ref{prop: main sol estimate} and proceed as before.
\end{proof}

\subsection{A variant of the lazy chain}
Here, we sketch the adjustments needed to deal with the original Markov chains appearing in \cite{CD20}. In the Markov chains considered there, the function $f$ is always applied, but $\varepsilon_n=0$ with positive probability. Note that the arguments in \cite{CD20} only work when $f$ is bijective, even though the Markov chain when $f(x)=x^2$ still makes sense, so we restrict to the setting where $f\in\mathcal{B}(p,d)$.

\begin{theorem}
\label{thm: bij other lazy}
Let $p$ be a prime and $d\in\mathbf{N}$. Let $f\in\mathcal{B}(p,d)$, and let $\gamma\in\mathbf{F}_p$ be non-zero. Let $\varepsilon_i$ denote independent random variables uniform on $\{-\gamma, 0, \gamma\}$. Consider the Markov chain defined by
\begin{equation*}
    X_{n+1}=f(X_n)+\varepsilon_{n+1}.
\end{equation*}
Then if $\pi$ denotes the uniform distribution on $\mathbf{F}_p$, for any $\varepsilon$, there is some constant $C=C(\varepsilon,d)$ such that
\begin{equation*}
    \sup_{X_0\in\mathbf{F}_p}\|\Prob(X_n\in\cdot)-\pi\|_{TV}\leq e^{-C\frac{n}{p^{1+\varepsilon}}}.
\end{equation*}
\end{theorem}

By the arguments of Section 2 in \cite{CD20}, it suffices to bound the second largest eigenvalue of a symmetrization using Cheeger's inequality, which reduces to showing that
\begin{equation*}
    \frac{e(S,S^c)}{\min(|S|,|S^c|)}
\end{equation*}
is large in the graph determined by the random walk $Pf^{-1}P^2fP$, where $P$ is the transition matrix for taking a random step of either $0$ or $\pm \gamma$ with equal probability, and $f$ is the permutation matrix corresponding to the bijection $f$. Actually to make an argument later work, we instead work with $PfP^2f^{-1}P$, which has the same (non-zero) spectrum.

Now certainly, it cannot hurt to throw away edges, and so we instead consider the bijective function $g(x)=f(f^{-1}(x)+\gamma)$, and consider the graph with edges connecting $x$ to $g(x)\pm \gamma$ (we are forcing a lazy first step from $P$, and after appling $f^{-1}$ one lazy step from $P$ and then adding $\gamma$).

Now from the arguments in Section \ref{sec: pf of bij}, if the set $S$ has a $2\gamma$-AP decomposition with too many intervals, then $e(S,S^c)$ will be large since every interval contributes an edge, and otherwise we can apply the exponential sum bounds of Section \ref{sec: exp sum bounds}, with an analogous version of Lemma \ref{lem: complete estimate rational}. We first need the following result to apply the Weil bound in this situation.

\begin{lemma}
\label{lem: not const}
Let $P,Q\in\mathbf{F}_p[x]$ be coprime and of degree at most $p/4$, and suppose that $P(x)/Q(x)$ is not constant or linear. Then for all $\alpha,\beta,\gamma\in\mathbf{F}_p$ with $\beta,\gamma\neq 0$, the function
\begin{equation*}
    \alpha P(x+\gamma)/Q(x+\gamma)+\beta P(x)/Q(x)
\end{equation*}
is not a constant function.
\end{lemma}
\begin{proof}
Suppose that $\alpha P(x+\gamma)/Q(x+\gamma)+\beta P(x)/Q(x)=c$ away from the poles of $Q(x)$ and $Q(x+\gamma)$. Then $\alpha P(x+\gamma)Q(x)+\beta P(x)Q(x+\gamma)=cQ(x)Q(x+\gamma)$ for more than $p/2$ points, and so both sides must be equal as polynomials since they are polynomials of degree at most $p/2$. Then as $P$ and $Q$ are coprime, $Q(x)$ must divide $Q(x+\gamma)$ and $Q(x+\gamma)$ must divide $Q(x)$, and by comparing the coefficient of the highest degree term it's clear that $Q(x)=Q(x+\gamma)$, so $Q(x)$ is a constant. Then $\alpha P(x+\gamma)+\beta P(x)=c$, and if $P$ is not a constant, this implies $\alpha=-\beta$ by comparing the highest degree term. Then the coefficient of the second highest degree term on the left is non-zero, so $P$ must be a linear function.
\end{proof}

Now, we prove the analogue of Lemma \ref{lem: complete estimate rational}.

\begin{lemma}
Let $f\in\mathcal{B}(p,d)$ and let $g(x)=f(f^{-1}(x)+\gamma)$. Then for any interval $I$, and any $k\in\mathbf{F}_p^\times$,
\begin{equation*}
    \sum_{n=1}^p\left|\sum_{x\in I}e_p(kg(x+n)) \right|^2\leq C p|I|,
\end{equation*}
where the constant $C$ depends only on $d$.
\end{lemma}
\begin{proof}
Since the constant $C$ can depend on $d$, if we take $C$ to be at least $4d$, then if $d\geq p/4$, the statement is trivially true as the left hand side is at most $p|I|^2\leq p^2|I|$. Thus, we assume that $d<p/4$.
	
Proceeding as in the proof of Lemma \ref{lem: complete estimate rational}, we have
\begin{equation*}
\begin{split}
    &\sum_{n=1}^p\left|\sum_{x\in I}e_p(kg(x+n)) \right|^2
    \\=&\sum_{x,y\in I}\sum_{\alpha=1}^p\sum_{a,b=1}^p p^{-1} e_p(k(a-b))e_p(\alpha(g^{-1}(a)-g^{-1}(b)-x+y))
    \\=&p^{-1}\sum _{\alpha=1}^p\left|\sum_{a=1}^p e_p(\alpha f(a-\gamma)+k f(a))\right|^2\left|\sum_{x\in I}e_p(\alpha x)\right|^2,
\end{split}
\end{equation*}
where the substitution $a\mapsto f(a)$ is made. Then as $\alpha f(a-\gamma)+kf(a)$ is not constant in $a$ by Lemma \ref{lem: not const}, we can apply Lemma \ref{lem: weil bound}. The proof then proceeds as before.
\end{proof}
The rest of the arguments in Section \ref{sec: exp sum bounds} are unchanged, leading to a version of Proposition \ref{prop: main sol estimate} for the function $g$, which completes the argument.

\section*{Acknowledgements}
This research was supported in part by NSERC. The author would like to thank Steve Butler, Alex Cowan, Persi Diaconis, and Kannan Soundararajan for their help. The author would also like to thank an anonymous referee for many suggestions which greatly improved the exposition.

\bibliography{bibliography.bib}
\bibliographystyle{amsplain}

\end{document}